\documentclass[12pt,a4paper,reqno]{amsart}

\usepackage[utf8]{inputenc}
\usepackage[margin=1in,footskip=0.25in]{geometry}
\usepackage{amsmath, amsthm, amssymb}
\usepackage{graphicx}
\usepackage{color}
\usepackage{overpic}
\usepackage{enumerate}
\usepackage{hyperref}

\newcommand{\R}{\mathbb{R}}

\newcommand{\N}{\mathbb{N}}

\newtheorem{teo}{Theorem}[section]

\newtheorem{rem}[teo]{Remark}
\newtheorem{lema}[teo]{Lemma}

\newtheorem{defi}[teo]{Definition}
\newtheorem{prop}[teo]{Proposition}
\newtheorem{cor}[teo]{Corollary}

\newcommand{\re}{\mathbb{R}}

\date{}

\title[The period function with applications]{The local period function for Hamiltonian systems with applications.}

\author[C. A. Buzzi]{Claudio A. Buzzi}
\address{Mathematics Department, Universidade Estadual Paulista Julio
de Mesquita Filho, 15054-000 S\~ao Jos\'{e} do Rio Preto, S\~ao Paulo,
Brazil } \email{claudio.buzzi@unesp.br}

\author[Y. R. Carvalho]{Yagor Romano Carvalho}
\address{Mathematics Department, Universidade Estadual Paulista Julio
de Mesquita Filho, 15054-000 S\~ao Jos\'{e} do Rio Preto, S\~ao Paulo,
Brazil} \email{yagor.carvalho@unesp.br}

\author[A. Gasull]{Armengol Gasull}
\address{Mathematics Department, Universitat Aut\`{o}noma de Barcelona,
08193 Bellaterra, Barcelona, Catalonia, Spain}
\email{gasull@mat.uab.cat}


\begin{document}

    \begin{abstract}
In the first part of the paper we develop a constructive procedure
to obtain the Taylor expansion, in terms of the energy, of the
period function for a non-degenerated center of any planar analytic
Hamiltonian system. We apply it to several examples, including the
whirling pendulum and a cubic Hamiltonian system.  The knowledge of
this Taylor expansion of the period function for this system  is one
of the key points to study the number of zeroes of an Abelian
integral that controls the number of limit cycles bifurcating from
the periodic orbits of a planar Hamiltonian system that is inspired
by a physical model on capillarity. Several other classical tools,
like for instance Chebyshev systems are applied to study this number
of zeroes. The approach introduced can also be applied in other
situations.
    \end{abstract}

    \maketitle

    \section{Introduction and main results}

Let $\gamma_s,$ with $s\in I\in\R,$ be a parameterized  continua of
periodic orbits of a planar autonomous differential system. In
general, $I$ is either an open interval or an interval of the form
$[s_0,s_1).$ The function that assigns to each $s$ the minimal
period of $\gamma_s$ is called {\it period function} and it is
denoted by $T(s).$ Similarly, the function that assigns to each
$\gamma_s$ the area surrounded by this closed curve is denoted by
$A(s)$ and  called {\it area function.} The period function is
 important to study theoretical properties of
planar ordinary differential equations and their perturbations, see
for instance \cite[ pp. 369-370]{ChoHal1982}; to understand some
mathematical models in physics or ecology, see \cite{ConVil2008,
FreGasGui2004, Rot1985, Wal1986} and the references therein;   in
the description of the dynamics of some discrete dynamical systems,
see \cite{BeuCus1998, CimGasMan2008, CimGasMan2015}; or for counting
the solutions of some boundary value problems, see \cite{Chi1987,
Chi1988}. When the system is Hamiltonian, with Hamiltonian function
$H$ and $\gamma_h\subset\{H=h\},$ it is natural to consider $s=h$
and write $T=T(h).$

Given a planar analytic Hamiltonian system
\begin{equation}\label{eq:ham}
\dot{x}= -H_y(x,y),\qquad \dot{y}=H_x(x,y),
\end{equation}
 with a non-degenerated
center at the origin (that without loss of generality we will
associated to $h=0$ and then $I=[0,h_1)\subset\R$)  it is known that
$T(h),$ in a neighborhood of $h=0,$ is an analytic function of the
energy $h$ and it is given by the derivative with respect $h$ of the
area function $A(h),$ see \cite{ManVil2002}. There are several
authors that compute the Taylor series of $T$  at $h=0$ for
particular Hamiltonian systems but, to the best of our knowledge,
most examples deal with Hamiltonian functions with separated
variables $H(x,y)=F(x)+G(y),$ see for instance
\cite{Bel2011,Fos2004} and their references. Our first result
provides a systematic constructive approach for finding this Taylor
series up to any order for any Hamiltonian system.

\begin{teo}\label{th:main1}
Let $ H $ be an analytic function with $ H (0,0) = 0 $ and assume
that the Hamiltonian
         system \eqref{eq:ham} has a non-degenerate center at the
         origin. Then:

         \begin{enumerate}[(i)]

         \item  In a neighborhood of $ h = 0 $ the period function $T(h)$ and the area function $A(h)$ are analytic
           and satisfy $A'(h) = T(h).$

          \item Let $\top(\rho),$ with $\rho>0,$ be the period of the orbit
          starting at the point $(x,y)=(0,\rho).$ Then:
\begin{itemize}
          \item The function  $\top$ is
          analytic at $\rho=0$ and its Taylor development  at $\rho=0$  can be
          obtained algorithmically from the expression of \eqref{eq:ham} in polar
          coordinates.
          \item  The equation $H(0,\rho)=h$ has, in a positive neighborhood of $h = 0,$  a
          positive solution $\rho= S(\sqrt{h})$ with $S$ analytic at
          zero. Moreover  its Taylor development at zero  can also be obtained algorithmically from the one of
          $H(0,\rho).$

          \item It holds that $T(h)= \top(S(\sqrt{h}\,)).$ Then
          $A(h)=\int_0^h T(s)\,ds.$

 \end{itemize}

 \end{enumerate}

\end{teo}

We want to stress that our contribution restricts to item (ii). As
we have already said, item (i) is proved in \cite{ManVil2002}. We
put both together for the sake of clarity. We remark that if the
system is only of class $\mathcal{C}^k,$ for some $k\in\N,$ our
approach can be adapted to this setting providing several terms of
the Taylor expansion of $T(h)$ and $A(h).$

It is also  worthwhile  to comment that sometimes it is also
possible to obtain a closed integral   expression for $T(h)$ or to
prove that it satisfies some differential equation.  Hence other
totally different approach consists in studying this integral and
try to obtain its Taylor series at the origin or to study a
particular solution of this differential equation. As far as we
know, our two steps procedure is new, totally different and
applicable to all integrable planar systems with a non-degenerated
center. It is tedious, but so systematic that can be used with any
computer algebra system. In this paper, these computations are done
with Maple.

In Section \ref{se:s2} we will apply it to several examples,
including the whirling pendulum and a family of quadratic
Hamiltonian systems. Our main application is given in the second
part of this paper that studies, with this point of view, the number
of limit cycles that appear in the study of a planar  system
motivated by a physical model on capillarity that we briefly
describe.

Capillary action is the physical property that fluids have to go
down or up
    in extremely thin tubes.
     The capillary rise in a narrow vertical tube is
     a remarkable physical phenomenon that can also be observed in many other everyday situations,
      such as water transport in the soil or plants.  The increased use of capillary flow as an
       application in the industry has made a substantial growth in the search for more appropriate mathematical models.
       For the description of the model and more details see \cite{PloSwi2018}.  In that paper we
       can see this model is
\begin{equation}\label{sistemakreal}
    \left\{ \begin{array}{l}
    u'=v,\\
    v' = 1 + K v - \sqrt{2u},
    \end{array} \right.
    \end{equation}
    and it is defined on $u>0,$ where the prime denotes a differentiation with respect to the real time variable $t$ and
    $K\in\R$. Notice that \eqref{sistemakreal} has a unique equilibrium point at $\left( \frac{1}{2},0
    \right)$. Moreover when $K=0$, it is a Hamiltonian system with
 $H(u,v)=-u+\sqrt{8u^3}/3+v^2/2$.
     Its critical point is a local minimum, so it is of  center  type.

Consider the following perturbation of the Hamiltonian system
(\ref{sistemakreal}), with $K=0$, motivated by the appearance of
$\sqrt{2u}$ in its expression,
    \begin{equation}\label{sistemaintegralabeliana}
    \left\{ \begin{array}{l}
    u' = v + \varepsilon \widetilde{P}(\sqrt{2u}),   \\
    v' =  1 - \sqrt{2u} + \varepsilon \widetilde{Q}(\sqrt{2u}),
    \end{array}
    \right.
    \end{equation}
    defined on $u>0,$ where $\widetilde{P}(\sqrt{2u})$ and $\widetilde{Q}(\sqrt{2u})$ are
    polynomials in the variable $\sqrt{2u}$ and $\varepsilon \in \re$ is a small parameter.
     Now we perform the change of variable $x=\sqrt{2u}-1$ and $y=-v$, and a time rescaling. So, system \eqref{sistemaintegralabeliana} becomes
    \begin{equation}\label{sistemaqueaplicamosintabel}
    \left\{ \begin{array}{l}
    x' = -y + \varepsilon P(x), \\
    y'=  x + x^2 + \varepsilon  Q(x),
    \end{array} \right.
    \end{equation}
     defined on $x+1>0,$ with $P$ and $Q,$ again polynomials with respective degrees $n+1$ and $m.$ Clearly, when  $\varepsilon=0$ the above
     system is a Hamiltonian system, with a center at the origin, and
    \begin{equation}\label{hamsistemaintabel}
    H(x,y)= \frac{x^2}{2} + \frac{x^3}{3}+ \frac{y^2}2,
    \end{equation}
    the energy levels $\{H(x,y)=h\}\cap\{x+1>0\}$ of the Hamiltonian (\ref{hamsistemaintabel}) for $h \in (0, \frac{1}{6})$ are ovals.
By using the classical approach, see Subsection \ref{ss:whp} for
more details, the number of limit cycles that bifurcate from the
periodic orbits $\gamma_h,$ by a first order analysis in
$\varepsilon$ for system \eqref{sistemaqueaplicamosintabel} is given
by the maximum number of simple zeroes of the Abelian integral
associated to \eqref{sistemaqueaplicamosintabel},
\begin{equation}\label{eq:final}
I(h)=\int_{\gamma_h} P(x)\,dy=-\int_{\gamma_h} P'(x) y\,
dx=\sum_{j=0}^{n} \alpha_j I_{j}(h)\quad\mbox{with}\quad I_j(h)=
\displaystyle \int_{\gamma_h} x^j y\,  dx,
\end{equation}
where the parameters $\alpha_j$ are arbitrary real parameters that
depend on the ones of $P.$

 A first result relates the above integrals with the first part
of our paper. It will be proved in Subsection \ref{ss:raai}.

\begin{prop}\label{pr:AT}  Let $A(h)$ and $T(h)$ be the area and
period functions associated to the Hamiltonian system with
$H(x,y)=x^2/2+y^2/2+x^3/3.$
        \begin{enumerate}[(i)]
\item  The functions $A$ and $T$ verify the $2\times2$ Picard-Fuchs
equations
\begin{equation*}
6(6h -1)h \left(\begin{array}{c}
 A'(h)  \\
T'(h)
\end{array}\right)=\left(
\begin{array}{cc}
0 & 6(6h -1)h\\
-5 & 0\\
\end{array}
\right)\left(
         \begin{array}{c}
           A(h)\\
           T(h)\\
         \end{array}
       \right).
\end{equation*}

\item The function $A$ verifies the Hill's equation
\begin{equation}\label{eq:hill-AT}
A''(h)=\frac5{6(1-6h)h} A(h).
\end{equation}

\item The function $p(h)=A(h)/T(h)$ satisfies the Riccati
differential equation
\begin{equation*}
p'(h)=-p^2(h)+\frac5{6(1-6h)h}.
\end{equation*}

\item Consider the  Abelian integrals defined in
        \eqref{eq:final}. Then $I_0(h)=-A(h)$ and
 for all $0\le n\in\N,$
\begin{align*}
I_{3n+1}(h)=& a_{n+1}(h)A(h)+b_{n}(h) (6h-1)h T(h),\\
I_{3n+2}(h)=& c_{n+1}(h)A(h)+d_{n}(h) (6h-1)h T(h),\\
I_{3n+3}(h)=& e_{n+1}(h)A(h)+f_{n}(h) (6h-1)h T(h),
\end{align*}
where $g\in\{a,b,c,d,e,f\},$ and $g_k\in\mathbb{Q}(h)$ denotes a
polynomial of degree $k.$ Moreover, $I_{3n+2}$ is a linear
combination of several $I_j$ with $j<3n+2$ and $j\not\equiv2 \pmod
3.$
\end{enumerate}
\end{prop}

We prove next theorem, see Subsection \ref{ss:ect} to recall the
definition of extended complete Chebyshev (ECT) system.

\begin{teo}\label{th:main2}   Set $I(h)$ given in \eqref{eq:final},
where $\gamma_h,$ for $h\in L:=(0,1/6),$ are the ovals of
$\{x^2/2+x^3/3+y^2/2=h\}.$ Then,
\[
I(h)=\sum_{j=0}^{n} \alpha_j I_{j}(h)=\sum_{j=0,\, j\not\equiv 2
\!\pmod 3}^{n} a_j I_{j}(h)
\]
where $a_j$ can be taken  free real parameters and depend on the
coefficients of $P.$ Moreover:
\begin{enumerate}[(i)]

\item The functions $I_j(h), j\not\equiv 2 \pmod 3,$ are linearly
independent. In particular,  there are values of the constants $a_j$
such that $I(h)$ has $N(n)$ simple zeroes in $L$ where
\[
\qquad N(n)=\operatorname{Card}\{j\,:\, 0\le j\le n\, \mbox{ and }\,
j\not\equiv 2 \pmod 3\}-1=n-[(n+1)/3]
\]
and, as usual, $[\cdot\,]$ denotes the integer part function.

\item  There exists  $h_1>0$ such that for each $n\le 50$ the functions $I_j(h),$ $0\le j\le n$ and  $j\not\equiv 2 \pmod
3,$ form an ECT system  on $(0,h_1).$ In particular, for $n\le50$
the maximum number of zeroes of $I(h)$ in $(0,h_1),$ taking into
account their multiplicities is $N(n).$

\item The functions $\big(I_0(h), I_1(h)\big)$ and $\big(I_3(h), I_1(h), I_0(h)\big)$ form
 ECT systems on $L.$ In particular, for $n=0,1,2,3,$ the maximum
number of zeroes of $I(h)$ in $L,$ taking into account their
multiplicities is $N(n).$

\end{enumerate}
\end{teo}

A similar result could also be proved taking in
\eqref{sistemaqueaplicamosintabel} more general perturbations with
$P$ and $Q$ depending also on $y.$ In fact, while we were ending the
above proof we realized that this more general  problem was
addresses many years ago by Petrov with an equivalent expression of
the Hamiltonian (he took $H(x,y)=y^2+x^3-x$), see \cite{Petrov1988},
or the second part of the book \cite{ChrLi2007} with the also
equivalent expression $H(x,y)=x-x^3/3+y^2/2.$ He got a more general
result by using the complexification of the corresponding $I_j(h),$
the Picard-Fuchs equations satisfied for them and the argument
principle. In particular he proved that the maximum number of zeroes
of $I(h)$ in $L,$ taking into account their multiplicities is always
$N(n).$

Our proof of item (i) is similar to the one of Petrov but the proof
of item (ii) strongly uses that knowing the function $T(h)$ near
$h=0$ suffices to get lower bounds of the number of zeroes of $I(h)$
and gives a different computational approach to the problem, that is
valid for a given $n$.  We stop at $n=50,$ but  it is easy to
go further in our computations.

In fact, although Petrov approach gives strong results in this case,
as we will explain in Remark \ref{re:final}, our point of view can
also be used for studying  several perturbations of many Hamiltonian
systems. In all these cases a lower bound of the number of limit
cycles follows from the computation of the Wronskian of some
polynomials on $h$ at $h=0$ and these polynomials can be obtained
simply from the knowledge of the Taylor's series $T(h)$ or $A(h)$ at
$h=0.$ The key point for proving item (iii) is the method introduced
in \cite{GraManVil2011,ManVil2002}. It provides an alternative
approach to the one of Petrov for small $n.$

A straightforward corollary of Theorem \ref{th:main2} is:

    \begin{cor}\label{corolcotamaxima01} For $\varepsilon$ small enough, system \eqref{sistemaqueaplicamosintabel}
has at least $n-[(n+1)/3]$ limit cycles surrounding the origin, that
bifurcate from their periodic orbits $\gamma_h,$ $h\in(0,1/6).$
    \end{cor}

    The paper is organized as follows: In Section \ref{se:s2} we
    prove Theorem \ref{th:main1} and apply it to several Hamiltonian
    systems. In Section \ref{sec:2} we include same preliminaries
    devoted to prove Theorem \ref{th:main2} and we also prove Proposition
    \ref{pr:AT}. More concretely, there
    is a subsection devoted to recall the relation between limit cycles and Abelian integrals; a second
    one dedicated to find the Picard-Fuchs, Hill and Riccati
    differential equations for  $I_0(h)$ and $I_1(h),$ see Proposition
    \ref{pr:pf}, and
    to present other relations among all the involved Abelian integrals; a third
    one about the parameterization of genus 0 planar polynomials
    curves and its application to our problem; and the last one on Chebyshev systems and how to use them in our
    situation. Finally, in Section \ref{sec:3} we prove  Theorem \ref{th:main2} and Corollary \ref{corolcotamaxima01}.

\section{Proof of Theorem \ref{th:main1} and some
applications}\label{se:s2}

Before proving Theorem \ref{th:main1} and for the sake of
completeness we prove  a preliminary result. Notice that the
solution of  an equation of the form $F_0(w,z)=
        \prod_{i=1}^n (w- \alpha_i z)=0,$ with all $\alpha_i$ different, is given by the  $n$ straight lines $w=\alpha_i z,$ $i=1,2,\ldots, n.$
         Next lemma asserts that when
        we consider a more general analytic equation of the form
$F(w,z)=b F_0(w,z)+O(n+1)=0,$  with $b\ne0,$ and the $O(n+1)$ part
denotes terms with degree at least $n+1$ in $w$ and $z,$ then its
solutions near $(0,0)$ is given by $n$ analytic branches $w=W_i(z),$
$i=1,2,\ldots, n,$ that are tangent to these $n$ lines. In fact,
next result and more general ones can be obtained and proved by
using the so called Newton's polygon, see  \cite{AndroLeontGor1973}
for more details.

    \begin{lema}\label{lemasoleqanalitica}
        Consider an analytic function  $F(w,z)=
        b\prod_{i=1}^n (w- \alpha_i z)+ O(n+1),$ with $b\ne0$ and all $\alpha_i$
        different. Then, in a neighborhood of $(0,0),$ the solutions
        of $F(w,z)=0$ are given by $n$ branches $w=W_i(z)= \alpha_iz+O(z^2),$ $i=1,2,\ldots, n,$ where all the
        functions $W_i$ are analytic at zero and their Taylor series at the origin can be obtained by implicit derivation.
    \end{lema}
    \begin{proof}
        To solve $F(w,z)=0,$ we divide it by $b$  and we make the blow up $w=uz.$ Then we get the following equivalent equation
        $$ z^n \displaystyle\prod_{i=1}^n (u - \alpha_i) + z^{n+1}h(u,z)=0 ,$$
        with $h(u,z)$ being an analytic function at $(0,0).$
       Hence it suffices to consider  \[G(u,z)= \prod_{i=1}^n (u - \alpha_i) + z h(u,z)=0.\]
Since, for all $j \in \{ 1,\ldots,n\},$ it holds that
\[
G(\alpha_j,0)=0\quad\mbox{and}\quad \dfrac{\partial}{\partial u} G
\left(\alpha_j ,0 \right)= \displaystyle\prod_{i=1, \; i \neq j}^n
(\alpha_j - \alpha_i ) \neq 0,
\]
by the Implicit   Function Theorem it follows that for every $j$
there is an analytic function $U_j$ in variable $z$,
        that satisfies $U_j(0)=\alpha_l$ and $G(U_j(z),z)=0$ for all $z$ in a neighborhood of $0.$
        As $ w = uz $ in a neighborhood of $(0,0),$ the solutions of
        $F(w,z)=0$ are
        $w=W_j(z) = zU_j(z)= {\alpha_j}z  + O(z^2),$
        and the Taylor series of each of them can be obtained simply
        by implicit derivation.
    \end{proof}

\begin{proof}[Proof of Theorem \ref{th:main1}] (i) This result is
proved in \cite{ManVil2002}.

(ii) We will prove that the function $\top(\rho)$ is analytic at
$\rho=0$  for any non-degenerated center, not necessarily
Hamiltonian. We simply follow the approach developed in
\cite{AndroLeontGor1973}, see also \cite{GasGuiMan1997}.  It is not
restrictive to write the differential system as
    \begin{equation}\label{sistemaperiodo}
    \left\{ \begin{array}{l}
    \dot{x}= - y + f(x,y) ,\\
    \dot{y}= x + g(x,y),
    \end{array} \right.
    \end{equation}
    where $ f $ and $ g $ are analytic functions in a neighborhood of the origin starting with  terms at least of degree two.

    Passing system \eqref{sistemaperiodo} to polar coordinates $(r, \theta )$ we obtain $\dot{r} = S(r,\theta )$ and
    $\dot{\theta } = 1+ T_1(r,\theta )$. Now we leave $\theta $ as the new independent variable and so we get the
    following analytic differential equation
    \begin{equation}\label{eqdrdw}
    \dfrac{dr}{d\theta }= \dfrac{S(r,\theta )}{1 + T_1(r,\theta )} = F_2(\theta ) r^2 +\cdots+ F_{n}(\theta ) r^{n} + O(r^{n+1}).
    \end{equation}

    We consider the initial condition $ r (0) = \rho> 0.$
    We can write the solution of (\ref{eqdrdw}) with this initial condition as
    \begin{equation*}
    r(\theta ,\rho)=r_{\rho}(\theta )= \rho + u_1(\theta ) \rho + \cdots + u_{n}(\theta ) \rho^{n} + O(\rho^{n+1}),
    \end{equation*}
    which is also analytic. By plugging it in equation (\ref{eqdrdw}) we
    find
    each $ u_i (\theta ), \; i\ge1,$ by solving simple differential equations with initial conditions
    $u_i(0)=0.$ Notice also that $(1 + T_1(r,\theta ))^{-1}$ is
    analytic at $r=0$ because $T_1(0,\theta)\equiv0$ and so, in a
    suitable neighborhood of $r=0$
    \begin{equation*}
    \frac{1}{1+T_1(r,\theta )}=1 + \displaystyle\sum_{k=1}^{\infty} g_k(\theta ) r^k.
    \end{equation*}
  Next we consider the following differential equation
\begin{equation*}
    \dfrac{d t }{d\theta} = \frac1{1 + T_1(r(\theta ,\rho),\theta )} =1+ \displaystyle\sum_{k=1}^{\infty} g_k(\theta )\big( r(\theta,\rho)\big)^{k}
    =
1+ \displaystyle\sum_{k=2}^{\infty} G_k(\theta )\rho^{k-1} .
    \end{equation*}
Hence,
    \begin{align*}
    \top(\rho)=&\displaystyle\int_0^{\top(\rho)}
    dt =\displaystyle\int_0^{2 \pi} \dfrac{d\theta }{1+T_1(r(\theta ,\rho),\theta
    )}= \displaystyle\int_0^{2 \pi} \Big(1 +
\displaystyle\sum_{k=1}^{\infty} G_k(\theta )
    \rho^k \Big) d\theta\\ &= 2 \pi + \displaystyle\sum_{k=1}^{\infty} t_k
    \rho^k,\quad \text{where}\quad  t_k=\displaystyle\int_0^{2
    \pi} G_k(\theta ) d\theta,
\end{align*}
and we have used that $r(\rho,\theta)$ converges uniformly towards
$0$ when $\rho$ tends to $0.$

To prove the second item of statement (ii) we will use Lemma
\ref{lemasoleqanalitica} with $n=2.$ It is not restrictive to assume
that $H(x,y)=x^2/2+y^2/2+O(3),$ because otherwise a linear change
plus a rescaling of the time can be done before starting the study.
Hence there is a relation between $\rho>0$  and $h=z^2>0,$ given by
\[
F(\rho,z)=H(\rho,0)-h=H(\rho,0)-z^2=\frac{\rho^2}2-z^2+O(3)=
\frac12\big(\rho+\sqrt 2 z\big)\big(\rho-\sqrt2 z\big)+O(3)=0.
\]
Moreover, when $\rho>0$ and $h=z^2>0$ are such that $F(\rho,z)=0$
then  $T(h)=\top(\rho)$ because both values give the period of the
same periodic orbit. By Lemma \ref{lemasoleqanalitica}, near
$(0,0),$ the above equation has two analytic solutions
$\rho=S_j(z),$ $j=1,2$ where $S_j(z)=(-1)^j\sqrt2 z+O(z^2)$ satisfy
locally $F(S_j(z),z)\equiv 0.$ We are interested in $S:=S_2$ because
it sends positive values into positive ones. Hence $\rho=S\big(\sqrt
h\big),$ with $S$ analytic at $0$ and Taylor series computable
simply by implicit derivation, as explained in the proof of  Lemma
\ref{lemasoleqanalitica}.

Finally $T(h)= \top(\rho)=\top(S(\sqrt h)),$ as we wanted to prove.
Notice that although from this result it simply seems that $T$ is
analytic on $\sqrt h,$ from item (i) we already know that all the
odd terms of the Taylor series of $\top\circ S$ at zero must
cancell.
\end{proof}

\begin{rem}\label{re:rem1} Assume that the analytic system we are interested in writes as
    \begin{equation*}
\left\{ \begin{array}{l}
u'= - \alpha v + f_1(u,v) ,\\
y'= \beta u + g_1(u,v),
\end{array} \right.
\end{equation*}
where the prime denotes the derivative respect some time, say $s$,
and $\alpha$ and $\beta$ are both positive. Moreover, assume that it
is Hamiltonian, with Hamiltonian function $\widetilde{H}(u,v)=
\frac{\beta u^2}{2}
    + \frac{\alpha v^2}{2} +O(3).$ If we introduce the following
change of variables and time,
\begin{equation}\label{mudancavariaveis}
u=\frac{x}{\sqrt{\beta}}, \; v=\frac{y}{\sqrt{\alpha}} \quad
\mbox{and} \quad s=\frac{t}{\sqrt{\alpha \beta}},
\end{equation}
it writes as system \eqref{sistemaperiodo} and has the Hamiltonian
function $H(x,y)=\widetilde{H}
  \left( \frac{x}{\sqrt{\beta}},\frac{v}{\sqrt{\alpha}} \right).$  Then the period function with respect to time $t$ can be obtained in
  terms of the  energy levels, $ h,$ by the method developed in Theorem \ref{th:main1} and finally the time in the variable  $ s $
  is the previous one divided  by $ \sqrt {\alpha \beta }.$
\end{rem}

Next subsections are  dedicated to apply the above results to three
examples. In all the examples, by the sake of shortness, we only
present a few terms of the Taylor expansion of $T(h),$ but it is not
difficult to obtain much more terms. Obviously, the first terms of
$A(h)$ can be obtained from the ones of $T(h).$

\subsection{The whirling pendulum}

    In this example we calculate the first terms  in $h$ of the period function for the whirling pendulum.
      The motion of a whirling pendulum is considered for instance in
      \cite{Lic1999}. It writes as
\begin{equation*}
u''= -\dfrac{g}{\ell} \sin(u) +\omega^2 \sin(u) \cos(u), \quad u \in
\mathbb{S}^1,
\end{equation*}
where $\ell$ is the length of pendulum,  $u$ its angle deviation,
$g$ is the gravity constant and $\omega$ is a constant rotation
rate. Introducing a new variable $v=-u'$ converts this second order
equation  into  the planar analytic Hamiltonian system
\begin{equation}\label{eq:wp}
\begin{array}{l}
u'=-v, \\
y'= \sin(u) \left(a - b \cos(u) \right),
\end{array}
\end{equation}
where $a={g}/{\ell}>0$, $b=\omega^2 \geq 0,$ and Hamiltonian
function
\begin{align*}
\widetilde{H}(u,v)=& \frac{v^2}{2}- a \cos(u) + \dfrac{b}{2}
\cos^2(u) + a-\dfrac{b}{2}\nonumber\\
=& \dfrac{v^2}{2} + (a-b)\dfrac{u^2}{2} + \left(\dfrac{b}{6} -
\dfrac{a}{24}\right) u^4 + \left(\dfrac{a}{720} -
\dfrac{b}{45}\right) u^6 +  O(7).
\end{align*}
When $a-b>0$ we have a non degenerate center at the origin and this
is the case that we will consider. Following the notation of Remark
\ref{re:rem1} we apply the change of variables and time
\eqref{mudancavariaveis} with $\alpha=1$, $\beta=a-b$ and we obtain
\begin{equation*}
H(x,y)= \widetilde{H} \left( \frac{x}{\sqrt{a-b}},v
\right)=\dfrac{x^2}{2} + \dfrac{y^2}{2} + \left(
\dfrac{4b-a}{24(a-b)^2} \right) \dfrac{x^4}{4} + O(5).
\end{equation*}
Then, applying the first step of item (ii) of Theorem \ref{th:main1}
we obtain that
\begin{multline*}
\top(\rho) =2 \pi+ \frac{(a-4b)\pi}{8(a-b)^2}\rho^2+ \frac{(11a^2-16ab+176b^2)\pi}{1536(a-b)^4}\rho^4\\
+ \dfrac{(-11072b^3+173 a^3-2568 a b^2-708 a^2b)
\pi}{368640(a-b)^6}\rho^6 + O(\rho^8).
\end{multline*}
Doing the computations detailed in the second step of the same item
we arrive to
\begin{multline*}
S(\sqrt{h})=\sqrt{2}h^{1/2}-\frac{\sqrt{2}(a-4b)}{12(a-b)^2} h^{3/2} +\frac{\sqrt{2}(3a^2 - 16 ab +48b^2)}{160(a-b)^4}h^{5/2}\\
-\frac{\sqrt{2}(120ab^2 + 5a^3 - 320b^3 - 36
a^2b)}{896(a-b)^6}h^{7/2}+O(h^{9/2}).
\end{multline*}
Then
\begin{multline*}
T(h) = \top(S(\sqrt{h}))= 2 \pi + \dfrac{(a-4b)}{4(a-b)^2} \pi h +
\dfrac{3(3a^2 - 16 ab + 48b^2)}{128(a-b)^4} \pi h^2\\
+\dfrac{5(5a^3+120a b^2 -320 b^3-36a^2b)}{1024(a-b)^6} \pi h^3  + O(
h^4),
\end{multline*}
is the period function for the new time introduced in the change of
variables. Finally, we must divide it by $ \sqrt {\alpha \beta }
=\sqrt{a-b}$ to obtain the actual period function of system
\eqref{eq:wp},
\begin{multline*}
T(h) =  \frac{2\pi}{\sqrt{a-b}}\left(1+ \dfrac{(a-4b)}{4}
\left(\frac{h}{2(a-b)^2}\right) +
\dfrac{3(3a^2 - 16 ab + 48b^2)}{64} \left(\frac{h}{2(a-b)^2}\right)^2  \right.\\
 \left.+\dfrac{5(5a^3+120a b^2 -320 b^3-36a^2b)}{256} \left(\frac{h}{2(a-b)^2}\right)^3\right)  +   O ( h^4).
\end{multline*}
We can note if $b=0$ then the equation correspond to the simple
pendulum. Replacing $a=g/\ell$  gives
\begin{equation*}
T(h) =  2\pi\sqrt{\frac{\ell}{g}}\left(1+ \dfrac{1}{2^2}
\left(\frac{\ell h}{2g}\right) + \dfrac{3^2}{2^6} \left(\frac{\ell
h}{2g}\right)^2+\dfrac{5^2}{2^{8}} \left(\frac{\ell h}{2g}\right)^3
\right)+ O( h^4).
\end{equation*}
Notice that the above terms  coincide with the first ones of the
well-known expression of the period function of the pendulum given
by Lagrange
\[
T(h)=2\pi\sqrt{\frac{\ell}{g}} \sum_{n=0}^\infty
\left(\frac{(2n)!}{(2^n n!)^2}\right)^2 \left(\frac{\ell
h}{2g}\right)^{n},
\]
obtained from the expression of this period function in terms of an
elliptic integral.

\subsection{A quadratic system}

As a second application we take the simplest family of Hamiltonians
that are not of the form $H(x,y)=F(x)+G(y).$ More concretely we
consider the family of quadratic systems with Hamiltonian function
\begin{equation*}
H(x,y) = \dfrac{x^2}{2} + \dfrac{y^2}{2} -  \dfrac{x^3}{3} + ax y^2
- b \dfrac{y^3}{3}.
\end{equation*}
These systems are studied for instance in \cite{ArtLli1994,
HorIli1994}. Applying our two steps procedure we obtain that
\begin{multline*}
\top(\rho) = 2 \pi+ \frac{\pi}{6}(9a^2+5b^2-6a+5) \rho^2-
\frac{\pi}{9}(9a^2+5b^2-6a+5) \rho^3\\
+\frac {5\pi }{ 288}\left( 189\,{a}^{4}+378{a}^{2}{b}^{2}+77{b}^{4}-
180{a}^{3}-84a{b}^{2}+126{a}^{2}+10{b}^{2}-84a+77 \right)\rho^4
 + O(\rho^5),
\end{multline*}
\begin{equation*}
S(\sqrt{h})=\sqrt{2}h^{1/2}+\frac{2}{3}h
+\frac{5\sqrt{2}}{9}h^{3/2}+
\frac{32}{27}h^2+\frac{77\sqrt{2}}{54}h^{5/2}+\frac{896}{243}h^3+O(h^{7/2}),
    \end{equation*}
   and finally,
    \begin{align*}
    T (h) =\top(S(\sqrt{h}))= &  2 \pi + \frac{\pi}{3}(9a^2+5b^2-6a+5) h+
    \frac{5\pi}{72}\big(189{a}^{4}+378{a}^{2}{b}^{2}+77{b}^{4}\\
    &-180{a}^{3}-84a{b}^{2}
+126{a}^{2}+10{b}^{2}-84a+77 \big)h^2+O(h^3).
    \end{align*}

\subsection{System \eqref{sistemaqueaplicamosintabel}\label{ss:apl} with $\varepsilon=0$}

The   third example deals with the Hamiltonian system with
    \[
H(x,y)=\frac{x^2}2+\frac{y^2}2+\frac{x^3}3.
    \]
In this case, to prove item (iii) of Theorem \ref{th:main2} we need
to obtain more terms of the Taylor's development of $T(h)$. As in
the previous subsections we obtain first that
    \begin{equation*}
    \top(\rho) = 2 \pi+\frac{5}{6}\pi\rho^2+\frac{5}{9}\pi\rho^3+\frac{385}{288}\pi\rho^4
    +\frac{385}{216}\pi\rho^5
    +\frac{103565}{31104}\pi\rho^6+   \frac{85085}{15552}\pi\rho^7+
    \frac{6551545}{663552}\pi\rho^8+O(\rho^8),
    \end{equation*}
    and
    \begin{equation*}
    \rho_2(\sqrt{h})=\sqrt{2}h^{1/2}-\frac{2}{3}h +\frac{5}{9}\sqrt{2}h^{3/2}-\frac{32}{27}h^2+
    \frac{77}{54}\sqrt{2}h^{5/2} -\frac{896}{243}h^3+\frac{2431}{486} \sqrt{2}
    h^{7/2}+O(h^8).
    \end{equation*}
Then,
\begin{align}\label{eq:th}
T(h) =&\top(S(\sqrt{h}))=  2 \pi + \dfrac{5}{3} \pi h +
\dfrac{385}{72} \pi h^2 + \dfrac{85085}{3888} \pi h^3  +
\dfrac{37182145}{373248} \pi h^4\nonumber\\ &+
\dfrac{1078282205}{2239488} \pi h^5 +\frac{1169936192425}{483729408}
\pi h^6 +O (h^7).
\end{align}

It is very intriguing the appearance of the primorial function in
the coefficients of the Taylor's series of $T(h).$ Recall that if
$p_n$ denotes the $n$th prime number, then the primorial of $p_n$ is
denoted by $p_n\#$ and is $p_n\#=p_1p_2\cdots p_n.$ For instance
recall that the numbers $p_n\#+1$ play a key role in the proof of
Arquimedes of the existence of infinitely many prime numbers.
Computing some more terms of the expression of $T(h)$ get
\begin{align*}
T(h)=\pi&\left(2+\frac{5\#}{2\cdot 3^2}h+\frac{
11\#}{2^{4}3^{3}}h^2+\frac{17\#}{2^{5}3^{6}}h^3+\frac{
23\#}{2^{10}3^{7}}h^4+\frac{29\#}{2^{11}3^{8}}h^5
+\frac{5\cdot7\cdot31\#}{2^{14}3^{11}}h^6\right.
\\ &\left.\,\,+\frac{5\cdot 41\#}{2^{15}3^{12}7}h^7
+\frac{5\cdot 47\#}{2^{22}3^{14}7}h^8+ \frac{5\cdot 7\cdot
53\#}{2^{23}3^{18}}h^9 +\frac{7\cdot 11\cdot
59\#}{2^{26}3^{19}}h^{10}\right)+O(h^{11}).
\end{align*}

In fact, it is known that $T(h)$ is a monotonous increasing
function, defined for $h\in[0,1/6)$ and tending to infinity when $h$
goes to $1/6,$ see for instance \cite{FreGasGui2004}.

As we will prove in Proposition \ref{pr:AT}, it can be seen that the
area function $A,$ where $A'(h)=T(h),$ satisfies the Hill's equation
\eqref{eq:hill-AT},
\[
A''(h)=\frac5{6(1-6h)h}A(h), \quad  A(0)=0,\, A'(0)=2\pi.
\]
From this second order differential equation it is easier to obtain
more terms of the Taylor's series of $T$ at zero. We notice that
this result is computationally simpler that our general approach but
only works for some special Hamiltonian systems.

The above initial value problem can be solved in terms of some
hypergeometric function $_2 F_1.$ It holds that
\[
A(h)=2\pi h\, _2F_1\Big(\frac16,\frac56;2; 6h\Big),
\]
where recall that for $|z|<1,$
\[
_2F_1(a,b;c;z)=\sum_{n=0}^\infty \frac{(a)_n(b)_n}{(c)_n} \frac
{z^n}{n!},
\]
$(d)_n=d(d+1)(d+2)\cdots(d+n-1)$ is the Pochhammer symbol and
$(d)_0=1.$ This expression helps to understand the appearance of all
prime numbers in the coefficients of $T$ because all  prime numbers
are of the form $6k+1$ or $6k+5$ and all these factors appear in the
numerators of $(1/6)_n(5/6)_n.$

    \section{Definitions and preliminary results.} \label{sec:2}

    This section reviews some definitions and prove some results that we
    will be use to prove Proposition \ref{pr:AT} and Theorem \ref{th:main2}.

\subsection{Limit cycles and Abelian integrals}\label{ss:whp}

    The second part of the Hilbert's
     16th problem asks about the maximum number of limit cycles and their relative
         locations in planar polynomial vector fields.  It is one of the most
    famous and difficult open problems in mathematics, see \cite{Ily2002,Sma1998}. At the end of the last century
    there has been a very significant advance when Ilyashenko and \'Ecalle independently proved the Dulac problem which
        is the case of individual finitude, that is, the number of limit cycles of a given
        planar polynomial differential system is finite.  We
        address for a very particular case of a weaker version of Hilbert's 16th problem, the so called infinitesimal Hilbert's  problem,
         which asks about an upper bound for the number of zeros of a particular Abelian
         integral.

    Let $ X_H = (- H_y, H_x) $ be the planar Hamiltonian vector field associated to \eqref{eq:ham} and consider a
    perturbation given by $ X _ {\varepsilon} = X_H + \varepsilon Y $, where
     $ Y = (P, Q) $ with $ P $ and $ Q $  polynomials.  The Poincar\'{e}-Pontryagin
     functions or Melnikov functions of  order $ k \in \N $ are obtained from  the coefficients
    of the displacement function of the first return Poincar\'{e} map as a Taylor's
    series in the $ \varepsilon $ variable near $ 0 $, that is, if $ P _ {\varepsilon} $ is the first return Poincar\'{e}
     map of the planar system $ X _ {\varepsilon} $ then its displacement function is given by $ \Delta _ {\varepsilon} (h) =
      P _ {\varepsilon} (h) - h $ and it has a Taylor's series in the $ \varepsilon $ variable near $ 0 $ given by
    \begin{equation*}
    \Delta_{\varepsilon}(h) = \varepsilon M_1(h) + \varepsilon^2 M_2(h) +\cdots+ \varepsilon^k M_k(h)+O(\varepsilon^{k+1}),
    \end{equation*}
    which converges to small values of $ \varepsilon $. Thus, when $M_1(h)\equiv M_2(h)\equiv M_{k-1}(h)\equiv0,$  the  Poincar\'{e}-Pontryagin functions
     or Melnikov functions of order $ k \in \N $ is given by $ M_k (h),  \; k \in \N $. The Poincar\'{e}-Pontryagin
     Theorem ensures that
    $$M_1(h) =  \displaystyle\int_{\gamma_h}Q(x,y) dx- P(x,y) dy,$$
    and that from each simple root of the
    $M_1$ bifurcates a single hyperbolic limit cycle.  Moreover,  if there is an $h^*$ such that
     $M_1(h^*)=M_1'(h^*)=\ldots=M_1^{(m-1)}(h^*)=0$ and $M_1^{(m)}(h^*) \neq 0$ we have at most $m$ limit cycles bifurcating from
     $\gamma_{h^*}.$
     So the total number of the limit cycles, counting the multiplicities,  bifurcating from a bounded continua of periodic orbits is at most
     the  number of isolated zeroes, taking into account their multiplicities, of the Abelian integral
      $M_1(h)$. This is the way how  isolated roots of Abelian integrals are related with
      the number of limit cycles of perturbed Hamiltonian systems.
      It is costumary to consider $I(h)=-M_1(h)$ as the first
      Poincar\'{e}-Pontryagin function, that is
     \begin{equation*}
    I(h)= \displaystyle\int_{\gamma_h} \omega=\displaystyle\int_{\gamma_h}  P (x, y) dy - Q (x, y) dx  =
    \iint_{\operatorname{Int}(\gamma_h)} \frac{\partial P(x,y)}{\partial x}+\frac{\partial Q(x,y)}{\partial y}\,dxdy,
    \end{equation*}
    where in the last equality we have used Green's Theorem.

      Hence in a few  words, the number of isolated zeros of $I (h)$, counted
with their multiplicities, gives
       an upper bound for the number limit cycles of $X_{\varepsilon}$  generated from the  ovals of $H$ near
       $\varepsilon=0.$ Moreover, if all these zeroes are simple this
       number of zeroes gives rise to the same number of hyperbolic
       limit cycles for the perturbed system.
       For more details see for instance  \cite{ChrLi2007,UrbHoss2017}.

    Notice that again from Green's Theorem we know that for any
    $i,j\in\N,$
    \begin{equation*}
    \displaystyle\int_{\gamma_h} x^i y^j dy  = \displaystyle\iint_{\operatorname{Int}(\gamma_h)} i x^{i-1} y ^j
     \,dx dy = -\displaystyle\int_{\gamma_h} \dfrac{i}{j+1} x^{i-1} y^{j+1} dx,
    \end{equation*}
    so, considering  $ P $ and $ Q $  polynomials, the function $I(h)$ can always be written as the linear combination
    \begin{equation}\label{espvetsistemaT}
    I(h)=\sum_{k=0}^\ell \beta_k J_k(h),\quad\mbox{where}\quad J_k(h)=\displaystyle\int_{\gamma_h} x^{i_k}
    y^{j_k}\,dx,
    \end{equation}
    for some $\ell\in\N,$ where  $i_k,j_k\in\N$ and  all $ \beta_k $ depend on the coefficients of $P$ and $Q.$

In view of expression \eqref{espvetsistemaT} it is natural to study
the number of zeros of linear combinations of $\ell+1$ functions. If
all these functions are linearly independent it is not difficult to
find linear combinations with exactly $\ell$ simple zeros in any
given interval. On the other hand, when the Hamiltonian $H$ is
polynomial,  many times some of the involved functions have some
linear or functional relations, like for instance, the so called
Picard-Fuchs equations that include also the derivatives of $I_k.$
Once all these relations are taken into account there appear some
other functions, say $\hat J_k(h),$ also involving Abelian
integrals, and maybe other elementary functions, such that
\[
I(h)=\sum_{k=0}^\ell \beta_k J_k(h)= \sum_{k=0}^p \alpha_k \hat
J_k(h),
\]
for some $\ell\ge p\in\N.$ Then, the simplest situation is when
these $p+1$ functions form a so called extended complete Chebyshev
system, see Section \ref{sec:2} for more information. This will be
the case for the system considered in this paper.

 In the literature there are many works dealing with zeros of Abelian integrals, see again
 \cite{ChrLi2007,UrbHoss2017}
 and their references. Without the aim of being exhaustive
 and for completeness we list some other techniques used elsewhere to approach the problem.
        For example, in some works (see \cite{DumFreLiCheZhaZif1997, HorIli1994, Peng2002}) there is a
         study the geometrical properties of the so-called centroid curve using the fact that it verifies a
          Riccati equation (which is itself deduced from a Picard-Fuchs system).  On the other hand in
          \cite{GasLiWeiLliZha2002, GauGavIli2009, Lubomir2001, LubomirIliya2002,Petrov1988}, the authors use complex analysis and
          algebraic topology (analytic continuation, argument principle, monodromy, Picard-Lefschetz formula,
          \ldots). Other times it is proved that the $p+1$
          functions are a Chebyshev system with some accuracy $k,$
          meaning that the maximum number of limit cycles
          provided by the Abelian integral is $p+k,$ see for
          instance \cite{GasLazTor2012}.

To end this subsection we state a simple but useful general result,
proved in \cite{ColGasPro}, that will be used in the proof of item
(i) of Theorem \ref{th:main2}.

\begin{lema}\label{le:li} Set $L\subset R$ an open real interval and let
$F_j:L\to\R,$ $j=0,1,\ldots,N,$ be $N+1$ linearly independent
analytic functions. Assume also that one of them, say $F_k, 0\le
k\le N,$  has constant sign on $L.$ Then, there exist real constants
$c_j,$ $j=0,1,\ldots,N,$ such that the linear combination
$\sum_{j=0}^N c_j F_j$ has at least $N$ simple zeroes in $L.$
\end{lema}

\subsection{Relations among Abelian integrals}\label{ss:raai}

This subsection is devoted to find relations among the integrals
\begin{equation}\label{eq:int}
I_k(h)=\int_{\gamma_h} x^k y\,dx, \quad k=0,1,2,\ldots,
\end{equation}
where $\gamma_h,$ for $h\in (0,1/6),$ are the ovals of
$\{H(x,y)=x^2/2+x^3/3+y^2/2=h\}$ and their derivatives. In
particular we obtain the Picard-Fuchs equations satisfied by $I_0$
and $I_1$ and the Hill's equation satisfied by $I_0.$ All our
computations are rather standard and we do not give all the details,
see for instance \cite{ChrLi2007}.

\begin{lema}
Consider the Abelian
integrals defined in \eqref{eq:int}. Then,
\begin{enumerate}[(i)]
\item For all $1\le k\in\N,$
\begin{equation}\label{eqpropquerelacionaIkcomI1}
    \displaystyle\sum_{j=0}^{k} {k\choose j}
         3^j2^{k-j} ( 3k - j) I_{3k- j-1
        }(h)=0.
        \end{equation}
In particular,
\[
I_2=-I_1,\quad  I_5=-\frac{3I_3+5I_4}2,\quad
I_8=-\frac{9I_5+21I_6+16I_7}4.
\]
\item   For $3\leq k\in\N$,
\begin{equation}\label{eq:k}
(2k+5)I_k(h) +3(k+1) I_{k-1}(h) - 6(k-2)hI_{k-3}(h)=0.
\end{equation}
In particular,
\begin{equation}\label{eq:3-4}
I_3=\frac{-12I_2+6hI_0}{11},\quad I_4=\frac{-15I_3+12hI_1}{13}.
\end{equation}
\end{enumerate}

\end{lema}

\begin{proof}  (i)  Since $3x^2+3y^2+2x^3=6h,$ for any $k\ge1,$ we have
\begin{align*}
0=&\int_{\gamma_h} \big(6h-3y^2\big)^k\, dy=\int_{\gamma_h}
\big(3x^2+2x^3\big)^k\, dy= \sum_{j=0}^k {k\choose j} 3^j 2^{k-j}
\int_{\gamma_h}x^{3k-j}\,dy\\=& -\sum_{j=0}^k {k\choose j} 3^j
2^{k-j}(3k-j)\int_{\gamma_h}x^{3k-j-1} y\,dx= -\sum_{j=0}^k
{k\choose j} 3^j 2^{k-j}(3k-j)I_{3k-j-1}(h),
\end{align*}
where we have used Green's Theorem.

            \item[(ii)]  From $x^2/2+x^3/3+y^2/2=h$ we know that on
            $\gamma_h,$ $(x+x^2)dx+ydy=0.$ Hence, multiplying this
            equality for $x^{k-2}y,$  integrating and using again Green's Theorem, we obtain that
            \begin{align*}
0= &\int_{\gamma_h} \Big(\big(x^{k-1}+x^{k}\big)y\,dx+
\int_{\gamma_h} x^{k-2}y^2\,dy= I_{k-1}(h)+I_k(h)+ \int_{\gamma_h}
x^{k-2}y^2\,dy\\=& I_{k-1}(h)+I_k(h)-\frac{k-2}3 \int_{\gamma_h}
x^{k-3}y^3\,dx\\=& I_{k-1}(h)+I_k(h)-\frac{k-2}3 \int_{\gamma_h}
x^{k-3}\Big(2h-x^2-\frac23x^3 \Big)y\,dx\\=&
I_{k-1}(h)+I_k(h)-\frac{k-2}3\Big(2hI_{k-3}(h)-I_{k-1}(h)-\frac23I_k(h)
\Big)\\=&\frac{2k+5}9I_k(h)+\frac{k+1}3I_{k-1}(h)-\frac{2(k-2)}3 h
I_{k-3}(h).
\end{align*}
From the above equality the result follows.

\end{proof}

    \medskip

Next result gives a relation between the Abelian integrals and their
derivatives.

 \begin{lema}\label{propquerelacionaIkcomI0}
        Considering the  Abelian integrals defined in
        \eqref{eq:int}. Then:
        \begin{enumerate}[(i)]
            \item  It holds that $I_0(h)=-A(h)$ and $I_0'(h)=-T(h),$
             where $T$ is the period function associated to $H$ and
             $A(h)$ is the areal surrounded by the oval $\gamma_h.$

            \item  For $0\leq k\in\N$,
            \begin{equation}\label{eq:deriv}
            2I_{k+3}'(h) + 3I_{k+2}'(h) +
            3I_k(h)-6h I_k'(h)=0.
            \end{equation}

        \end{enumerate}
    \end{lema}

    \begin{proof} Recall that the Gelfand-Leray formula, see for
    instance \cite[Thm. 26.32]{IlyYak2008}, allows to compute easily
    the derivative of Abelian integrals under suitable regularity
    conditions. It asserts that
    \[
\frac{d}{dh} \int_{\gamma_h} \omega =\int_{\gamma_h} \eta,
    \]
provided that $d\omega=dH\wedge \eta.$ In particular, for $0\le
k\in\N,$ by taking $\omega=x^ky\,dx$  and $\eta= x^k/y\, dx,$ since
$dH=(x+x^2)dx+ydy,$ it holds that $d\omega=dH\wedge \eta.$ Hence,
\begin{equation}\label{eq:der}
I_k'(h)=\frac{d}{dh} \int_{\gamma_h} x^ky\,dx =\int_{\gamma_h}
\frac{x^k}y\,dx.
\end{equation}

(i) By item (i) of Theorem \ref{th:main1} we know that $A'(h)=T(h),$
where $A(h)$ is the area surrounded by $\gamma_h.$ Hence, by Green's
Theorem,
\[
A(h)= \iint_{\operatorname{Int}(\gamma_h)} dx\,dy= -\int_{\gamma_h}
y\,dx=-I_0(h).
\]
Therefore,
\[
T(h)=A'(h)=-I_0'(h)=-\int_{\gamma_h} \frac{1}y\,dx,
\]
where we have used \eqref{eq:der} for $k=0.$ In fact, for this
particular Hamiltonian, the above relation simply follows by using
the first differential equation of the Hamiltonian system,
$dx/dt=-y.$

(ii)    By using \eqref{eq:der} and the expression of $H(x,y)=h$  it
holds that
            \begin{equation*}
            6h I_k'(h)= \displaystyle\int_{\gamma_h}  \left( 2x^3+3{x^2} + 3{y^2} \right)\dfrac{x^k}{y} \;
             dx = 2I_{k+3}'(h) + 3I_{k+2}'(h) + 3I_{k}(h),
            \end{equation*}
as desired.
\end{proof}

\begin{prop}\label{pr:pf}  Considering the  Abelian integrals defined in
        \eqref{eq:int}. Then
        \begin{enumerate}[(i)]
\item  The functions $I_0$ and $I_1$ verify the  $2\times2$ Picard-Fuchs
equations
\begin{equation}\label{eq:pf}
6(6h-1)h \left(\begin{array}{c}
 I_0'(h)  \\
I_1'(h)
\end{array}\right)=\left(
\begin{array}{cc}
6(5h-1) & -7\\
6h & 42h\\
\end{array}
\right)\left(
         \begin{array}{c}
           I_0(h)\\
           I_1(h)\\
         \end{array}
       \right).
\end{equation}

\item The function $I_0$ verifies the Hill's equation
\begin{equation*}
I_0''(h)=\frac5{6(1-6h)h} I_0(h).
\end{equation*}

\item The function $p(h)=I_0(h)/I_1(h)$ satisfies the Riccati
differential equation
\begin{equation*}
p'(h)=\frac1{6(6h-1)h}\left(7p^2(h)+6(2h+1)p(h)+6h\right).
\end{equation*}

\item It holds that $I_2(h)=-I_1(h)$ and, for all $1\le n\in\N,$
\begin{align*}
I_{3n}(h)=&h p_{n-1}(h)I_0(h)+q_{n-1}(h) I_1(h),\\
I_{3n+1}(h)=&h r_{n-1}(h)I_0(h)+s_{n}(h) I_1(h),\\
I_{3n+2}(h)=&h u_{n-1}(h)I_0(h)+v_{n}(h) I_1(h),
\end{align*}
where $w\in\{p,q,r,s,u,v\},$ and $w_k\in\mathbb{Q}(h)$ denotes a
polynomial of degree $k.$ Moreover, $I_{3n+2}$ is a linear
combination of several $I_j$ with $j<3n+2$  and $j\not\equiv2 \pmod
3.$
\end{enumerate}
\end{prop}

\begin{proof} (i) By using that $I_2=-I_1,$ together with
\eqref{eq:3-4}, expression \eqref{eq:deriv} for $k=0$ and $k=1,$
write as
\begin{align*}
&6hI_0'(h)+I_1'(h)-5I_0(h)=0,\\
&6hI_0'(h)+6(2-11h)I_1'(h)+6I_0(h)+77I_1(h)=0.
\end{align*}
Solving this system with respect to $I_0'$ and $I_1'$ gives the
result.

(ii) From the first equation of \eqref{eq:pf} we get that
\begin{equation}\label{eq:I1}
I_1(h)=\frac{6(5h-1)I_0(h)-6(6h-1)h I_0'(h)}{7}.
\end{equation}
By plugging this expression in the second equation of \eqref{eq:pf}
we arrive to the Hill's equation.

(iii) It is a direct consequence of item (i).

(iv) The result follows by using induction on $n$ and equalities
\eqref{eq:k} taking in each step blocks of three integrals
$I_{3n},I_{3n+1}$ and $I_{3n+2}.$  For instance,
\begin{align*}
I_3(h)=&\frac{6}{11}hI_0(h)+\frac{12}{11} I_1(h), \\
 I_4(h)=&-\frac{90}{143}hI_0(h)+\Big(\frac{12}{13}h-\frac{180}{143}\Big)
 I_1(h),\\
 I_5(h)=&\frac{108}{143}hI_0(h)+\Big(-\frac{30}{13}h+\frac{216}{143}\Big)
 I_1(h).
\end{align*}
The final property for $I_{3n+2}$ is a simple consequence of
relation \eqref{eqpropquerelacionaIkcomI1}.
\end{proof}

\begin{proof}[Proof of Proposition \ref{pr:AT}]
Notice that this proposition simply restates Proposition \ref{pr:pf}
but changing $I_0(h)$ and $I_1(h),$ by the two functions  $A(h)$ and
$T(h).$ In fact, it suffices to use the results of Lemma
\ref{propquerelacionaIkcomI0},  $I_0(h)=-A(h)$ and $I_0'(h)=-T(h),$
and that the expression \eqref{eq:I1} reads as
\begin{equation}\label{eq:quo}I_1(h)=\frac 6 7(1-5h)A(h)+\frac 6 7(6h-1)h T(h).\end{equation} Then
all the results are simple computations.
\end{proof}

\subsection{Involutions and rational
parameterizations}\label{ss:param}

Let $ A $ be a smooth function with a minimum at $ x = 0.$ Then, it
has associated an involution $\sigma$ defined on some open interval
$\mathcal{K}= (x_l, x_r) \ni 0$ that satisfies $A(x)=A(\sigma(x)).$
Recall that a map $\sigma$ is called {\it an involution} if
  $\sigma \circ \sigma = \operatorname{Id}$ and $\sigma \neq
  \operatorname{Id}.$ By the results of \cite{GraManVil2011}
  (see Theorem \ref{criteriointabeliana}) this involution plays an important role when
  studying some Abelian integrals associated to the Hamiltonian
  $H(x,y)=A(x)+y^2/2.$ In our case $A(x)=x^2/2+x^3/3$ and hence
  $z=\sigma(x)$ is defined implicitly by
  \begin{equation}\label{eq:W}
\frac{x^2}2+\frac{x^3}3-\frac{z^2}2+\frac{z^3}3=(x-z)S(x,y)=0,\quad
\mbox{where}\quad S(x,y)= \frac {x+z} 2 +\frac{x^2+xz+z^2}3.
  \end{equation}
Solving $S(x,z)=0$  we get
\begin{equation}\label{eq:ivo1}
    z=Z^\pm(x)=\dfrac{1}{4} \left(-3 - 2 x \pm \sqrt{3(3 - 4 x - 4 x^2)}
     \right).
    \end{equation}
Then $\sigma=Z^+$ and $\mathcal{K}=(-1,1/2).$

 As we will see, when one wants to apply
Theorem \ref{criteriointabeliana} we need to control the sign of
functions of the form $R(x,\sigma(x)),$ where $R\in\R(x,y)$ is a
polynomial. In this situation it is very useful to introduce the so
called  rational parameterizations of algebraic curves. Given a
planar algebraic curve $R(x,y)=0,$ it is said that admits a rational
parameterization if there exist two non-constant rational functions
$u(s)$ and $v(s), s\in\R,$ such that $R(u(s),v(s))\equiv0.$
Cayley-Riemann's Theorem (\cite{Abh1988,AbhBaj1988}) ensures that
$R$ can be rationally parameterized if and only if its genus is
zero. Moreover, in such case there are effective methods to find a
parameterization, see for instance  \cite[Chap. 4\&5]{Sen2008}. In
particular, when $S$ is an irreducible quadratic polynomial, it has
genus 0 and it can be rationally parameterized. Next lemma gives one
of its parameterizations and other useful expressions for our
forthcoming computations. Its proof is straightforward.

\begin{lema}
 A rational parameterization of the algebraic curve
$S(x,z)=0$ with $S$ given in \eqref{eq:W} is
\begin{equation*}
x=u(s)=\frac{-3s(s-2)}{2(s^2-2s+4)},\quad
z=v(s)=\frac{-3s}{s^2-2s+4},
\end{equation*}
and $u$ and $v$ map bijectively $[0,1]$ into $[0,1/2]$ and $[-1,0],$
respectively. Moreover
\[
\sigma(u(s))=Z^+(u(s))=v(s)\quad \mbox{and}\quad
Z^-(u(s))=\frac{3(s-2)}{s^2-2s+4}
\]
\end{lema}

Let us illustrate how to use it and its advantages with respect
other approaches in a simple example that will be used later. Assume
that we want to prove that the function
\begin{equation}\label{eq:exx}
M(x)=1+x+\sigma(x),\quad x\in(0,1/2) \end{equation}
 does not vanish.

A first naive way consists in trying to find its solutions, plugging
the expression of $\sigma=Z^+,$ given in \eqref{eq:ivo1}. Then,
isolating the square root term and squaring in both sides we obtain
the polynomial equation $2x^2+2x-1=0$ that has the root $x_0=(\sqrt
3-1)/2$ in $(0,1/2),$ that in fact is not a solution of
\eqref{eq:exx}. So, this approach fails unless we discard this
spurious solution.

A second powerful approach consists on using resultants, see
\cite{Stu2002}. An advantage  is that it can be utilized for any
involution $z=\sigma(x)$ defined implicitly by a polynomial relation
$S(x,z)=0.$ This is the method used systematically in
\cite{GraManVil2011,ManVil2002}. In this case it reduces to prove
that the following resultant
\[
U(x)=\operatorname{Res}_z\big({\overline M}(x,z),S(x,z) \big),
\]
where ${\overline M}(x,z)=1+x+z,$ does not vanish for $x\in(0,1/2).$
Of course, for this simple ${\overline M}$ there is no need of doing
the resultant because ${\overline M}(x,z)=0$ is equivalent to
$z=-1-x,$ but for higher degree functions this general approach can
always be used. In this case $U(x)=(2x^2+2x-1)/6,$ and as in the
previous approach $U(x_0)=0$ and, as a consequence, we cannot assure
that $M(x)$ does not vanish on $(0,1/2).$

Finally, with the approach that we propose, we can prove our goal.
The only disadvantage is that it only works when $S(x,z)=0$ has
genus 0, but fortunately, this is the situation in the case we are
dealing with. Notice that to prove that $M(x)$ does not vanish on
$(0,1/2)$ it suffices to prove that
\[
{M}\big(u(s)\big)=\overline{M}\big(u(s),
v(s)\big)=1+u(s)+v(s)=-\frac{s^2+4s-8}{2(s^2-2s+4)}
\]
does not vanish for $s\in(0,1),$ result that trivially holds.

In fact, the reason why this third approach works in this case,
while the two previous ones do not, is simple. The first two
approaches consider simultaneous  the other branch $z=Z^-(x)$
defined  by $S(x,z)=0$ and this branch is not taken into account in
the third one. In fact,
\[
\overline{M}\big(u(s), Z^-(u(s)\big)=-\frac{s^2-8s+4}{2(s^2-2s+4)},
\]
and this function vanishes at $s=s_0=4-2\sqrt3\in(0,1)$ and
$u(s_0)=x_0.$

The reader interested to see more utilities of the rational
parameterizations in dynamical systems can take a look to
\cite{GasLazTor2019}.

We will use either the second or the third methods when we study the
sign of functions $R(x,\sigma(x)).$ In fact, both approaches lead to
a final polynomial in one variable in $\mathbb{Q}(x),$ $x\in(0,1/2)$
or $\mathbb{Q}(s),$ $s\in(0,1).$ It is well known that the control
of the zeroes of these polynomials in the respective intervals can
be done by computing their Sturm sequences, see for instance
\cite{Sto2002}. We briefly recall this method that we will
systematically use without giving the details.

    \begin{defi}[Sturm's sequence]
        A sequence $(f_0,\ldots,f_m)$ of continuous real functions on $[a,b]$ is called a
          Sturm's sequence for $f=f_0$ on $[a,b]$ if the following holds:
        \begin{itemize}
            \item [(a)] $f_0$ is differentiable on $[a,b]$.
            \item [(b)] $f_m$ does not vanish on $[a,b]$.
            \item[(c)] if $f(x_0)= 0$, $x_0 \in [a,b]$, then $f_1(x_0) f_0'(x_0) >0$.
            \item[(d)] if $f_i(x_0)= 0$, $x_0 \in [a,b]$, then $f_{i+1}(x_0) f_{i-1}(x_0) <0$, $i \in \{1,\ldots,m\}$.
        \end{itemize}
    \end{defi}
    \begin{teo}[Sturm's Theorem]
        Let $(f_0,\ldots,f_m)$ be a Sturm's sequence for $f=f_0$ on $[a,b]$ with $f(a)f(b) \neq 0$.
          Then the number of roots of $f$ in $(a,b)$ is equal to $V(a)-V(b)$, where $V(c)$ is
           the number of changes of sign in the ordered sequence $(f_0(c),\ldots,f_m(c))$, where zeroes are not taken into account.
    \end{teo}

A Sturm's sequence for any polynomial $f$ with simple roots can be
easily found by a small variation of Euclid's  algorithm for finding
the greatest common divisor, see again the classical book
\cite{Sto2002}.

\subsection{Chebyshev systems}\label{ss:ect}

     For the
    characterization of Chebyshev systems in an open interval we will use the following results which can be
     found in \cite{KarStu1966} and \cite{Mardesic1998}.

    \begin{defi}
        Let $ u_0,\ldots, u_ {n-1}, u_n $ be functions defined in an open interval $ L $ of $ \re $.
        \begin{itemize}
            \item [(a)] The set of functions $( u_i)_{i=0}^n$ form a  Chebyshev system, or for short  $T$-system, on $L$ if any
             nontrivial linear combination $a_0u_0 +\cdots+ a_n u_n$ has at most $ n $ isolated roots in $ L $.

            \item [(b)] The ordered set of functions $( u_i)_{i=0}^n$ form a complete Chebyshev system, or for short a  $CT$-system, on
             $L$ if $( u_i)_{i=0}^k$ form a $T$-system for all $k=0,1,\ldots,n$.

            \item[(c)]   The ordered set of functions $ (u_i ) _ {i = 0} ^ n $ form an extended complet Chebyshev system, or
            for short an $ ECT $-system, on $ L $ if any nontrivial linear combination $
a_0u_0 + \cdots+ a_k u_k $
            has at most $ k $ isolated roots in $ L $ counting multiplicity, for every $ k = 0,1,\ldots, n $.
        \end{itemize}
    \end{defi}

    Notice that  an  $ECT$-system on $L$ is also a  $CT$-system on $L$.
    \begin{defi}
        Let $u_0,\ldots,u_n$ functions that have derivatives until order $n$ on $L$.  The Wronskian of such functions in $x\in L$ is given by
        $$W(u_0,\ldots,u_n)(x) =  \left| \begin{array}{ccc}
        u_0(x) & \cdots & u_n(x) \\
        u_0'(x) &  \cdots & u_n'(x) \\
        \vdots &  \ddots & \vdots \\
        u_0^{(n)}(x)  & \cdots & u_n^{(n)}(x)
        \end{array} \right|.$$
    \end{defi}
The following result is the most common approach to prove that a set
of functions forms an $ECT$-system.
    \begin{lema}\label{lemasistemaTintervaloaberto}
        The ordered set of functions $(u_0, \ldots, u_n)$ forms an $ ECT $-system on $ L $
        if, and only if, for every $ k = 0, \ldots, n $, $ W (u_0, \ldots, u_k) (x) \neq 0 $ for every $ x \in L $.
    \end{lema}

    \begin{rem}
        If $(J_0, J_1,\ldots, J_n ) $ forms an $ECT$-system on $ L $ then $\sum_ {i = 0} ^ n \alpha_i J_i =0 $
        has the same roots bifurcation diagram that  $\sum_ {i = 0} ^ n \beta_i t ^ i =0$ for the simple  $ECT$-system
        $(1,t,\ldots,t^n).$
         In particular, the  coefficients $\alpha_i$ can be chosen  such that  $\sum_ {i = 0} ^ n \alpha_i J_i =0$ has $n$ simple roots in
         $L.$
    \end{rem}

Next result was developed by Grau, Ma\~{n}osas and Villadelprat
(\cite{GraManVil2011,ManVil2002})  and  is an extension of a
previous work of Li and Zhang \cite{LiZhang1996} where the authors
provided a sufficient condition for the monotonicity of the ratio of
two Abelian integrals. It allows to prove that a set of Abelian
integrals, of some special shape and  for a special type of
Hamiltonian system, form a Chebyshev system, simply proving that a
similar property is satisfied by the integrands. Next, we state a
version of Theorem B in \cite{GraManVil2011} adapted to our
interests.

Recall that when $ A $ has a minimum in $ x = 0 $ then the origin of
the Hamiltonian systems has a center. Moreover, $A$ has an
associated involution $\sigma$ such
  that $A(x)=A(\sigma(x))$ for $x\in(x_l,x_r)\ni0.$

\begin{teo}\label{criteriointabeliana}
    {\rm (\cite{GraManVil2011})}
        Let us consider the $n$ Abelian integrals
        \begin{equation*}
        J_k(h) = \displaystyle\int_{\gamma_h}  f_k(x) y^{2s-1}dx, \; 0<s \in \N,  \; k=0,\ldots,n-1,
        \end{equation*}
        where each $f_k(x)$ is an analytic function and, for each  $h \in (0,h_0)$, $\gamma_h$ is the oval
        surrounding the origin contained in the level set $\gamma_h=\{ A(x) +y^{2}/2 = h \}$. Let $\sigma$ be
         the involution associated to $A$, and define
        $$\ell_k(x) = \dfrac{f_k(x)}{A'(x)}-\dfrac{f_k(\sigma(x))}{A'(\sigma(x))}.$$
       If $(\ell_0,\ldots,\ell_{n-1})$ is a
        $CT$-system on $(0,x_r)$ and $s>2(n-2)$ then $(J_0,\ldots,J_{n-1})$ is an $ECT$-system on $(0,h_0)$.
    \end{teo}

 When condition $s>2(n-2)$ is not  fulfilled it is possible, in some situations,
 to obtain equivalent expressions of the Abelian integrals for which the corresponding  new ``$s$'' is large enough to
 verify the inequality, see next lemma.
 The procedure for obtaining these new Abelian integrals follows from Lemma 4.1 of \cite{GraManVil2011} and other tricks developed in that paper.
   For the sake of completeness we also include its proof.

\begin{lema}\label{lemaexpoente}
Let $\gamma_h$ be an oval inside the level set $\{A(x) +y^2/2 =
h\}.$
\begin{enumerate}[(i)]

\item If $F$ is a smooth function such that $F/A'$ is analytic
at $x=0,$ then  for  $s \in \N\cup\{0\}$,
$$ \displaystyle\int_{\gamma_h}  F(x) y^{s}dx =  \displaystyle\int_{\gamma_h}  \left( \dfrac{F(x)}{(s+2)A'(x)} \right)' y^{s+2}dx.$$

\item If $F$ is a smooth function such that $F\cdot A/A'$ is analytic
at $x=0,$ then  for  $s \in \N\cup\{0\}$,
$$ h\displaystyle\int_{\gamma_h}  F(x) y^{s}dx =  \displaystyle\int_{\gamma_h}
 \left(\left( \dfrac{F(x)A(x)}{(s+2)A'(x)} \right)'+\dfrac{F(x)}2\right) y^{s+2}dx.$$
\end{enumerate}
\end{lema}

\begin{proof} (i) Notice that
\begin{align*}
 0=\displaystyle\int_{\gamma_h}  d\left( \dfrac{F(x)}{(s+2)A'(x)}
 y^{s+2}\right)&= \displaystyle\int_{\gamma_h}  \left( \dfrac{F(x)}{(s+2)A'(x)}
 \right)'y^{s+2} dx + \displaystyle\int_{\gamma_h} +\dfrac{F(x)}{A'(x)}
 y^{s+1}\,dy\\&=\displaystyle\int_{\gamma_h}  \left( \dfrac{F(x)}{(s+2)A'(x)}
 \right)'y^{s+2} dx -\displaystyle\int_{\gamma_h} F(x)
 y^{s}\,dx,
\end{align*}
where in the last equality we have used that $A'(x)\,dx+y\,dy=0$ on
$\gamma_h.$

(ii) In this case,
\begin{align*}
 h\displaystyle\int_{\gamma_h}  F(x) y^{s}dx =& \displaystyle\int_{\gamma_h}\left(A(x)+\frac{y^2}2\right)  F(x) y^{s}dx=\int_{\gamma_h}
 A(x)F(x) y^{s}dx+\displaystyle\int_{\gamma_h}
\dfrac{F(x)}2 y^{s+2}dx \\=& \displaystyle\int_{\gamma_h}
 \left(\left( \dfrac{F(x)A(x)}{(s+2)A'(x)} \right)'+\dfrac{F(x)}2\right)
 y^{s+2}dx,
 \end{align*}
where in the last step we have used item (i).
\end{proof}

\section{Proof of Theorem \ref{th:main2} and Corollary \ref{corolcotamaxima01}} \label{sec:3}

\begin{proof}[Proof of Theorem \ref{th:main2}]

(i) We start listing the set of ordered Abelian integrals
\[
I_0(h),I_1(h),I_3(h),I_4(h),I_6(h),\ldots,
I_{3k-2}(h),I_{3k}(h),I_{3k+1}(h),I_{3k+3}(h),\ldots,I_m(h),
\]
where we have removed from the list of all functions $I_j, 0\le j\le
n,$ the ones  with $j\equiv 2 \pmod 3$ and $m=n$ unless $n\equiv 2
\pmod 3,$ case in which $m=n-1.$ For the sake of notation we will
denote the above functions by
\[
F_0(h),F_1(h),F_2(h),\ldots, F_{N(n)}(h),
\]
keeping the same order. Notice that then,
\begin{equation}\label{eq:Fs}
I(h)=\sum_{j=0}^{N(n)} c_j F_j(h),
\end{equation}
where $c_j$ can be taken as arbitrary constants.

We claim that these $N(n)+1$ functions are linearly independent.
Moreover, since $F_0(h)=I_0(h)=-A(h),$ the function $F_0$ does not
vanish for $h\in(0,1/6).$ Hence, we can apply Lemma \ref{le:li} and
item (i) follows.

Let us prove the claim. First of all, notice that $I(h)$ can be
written as
\[
I(h)= v_{[N(n)/2]}(h) I_0(h)+ w_{[(N(n)-1)/2]}(h) I_1(h),
\]
where $v_k$ and $w_k$ are arbitrary polynomials of degree $k,$
because, by using item (iv) of Proposition \ref{pr:pf}, we know that
each time that for a given $m$ we consider some new terms
$F_{2m}(h)$ and $F_{2m+1}(h)$ in $I(h)$ it is equivalent to the
appearance of  two new terms of the form $h^m I_0(h)$ and $h^m
I_1(h)$ in the expression of $I(h).$ Notice that for each $n,$
$I(h)$ is expressed as a linear combination  $N(n)+1$ functions of
the form $h^jI_0(h)$ and $h^kI_1(h),$ for suitable  $0\le j\le
[N(n)/2]$ and $0\le k\le [(N(n)-1)/2].$

In order to prove the claim, assume that we consider a linear
combination of them that gives identically zero. Then
\[
\frac{I_1(h)}{I_0(h)}\equiv
-\frac{v_{[N(n)/2]}(h)}{w_{[(N(n)-1)/2]}(h)}=:\frac{v(h)}{w(h)},
\]
or, in other words, we had that $I_1/I_0$ would be the rational
function  $v/w.$ On the other hand we know that at
\[
A(h)\sim A_0 \quad\mbox{and} \quad T(h)\sim
k\ln(1-6h)\quad\mbox{when}\quad h\uparrow 1/6,
\]
for some $0<k\in R.$ This is so for the area function $A(h)$ because
$A_0$ is the area surrounded by the homoclinic loop contained in
$\{x^2/2+x^3/3+y^2/2=1/6\}$ and for the period function $T(h),$
because we know that $\lim_{h\uparrow 1/6}T(h)=\infty$ and its dominant
asymptotic term is given by the passage time near the hyperbolic
saddle $(-1,0)$ of the system \eqref{sistemaintegralabeliana} with
$\varepsilon=0,$ see for instance \cite{GasManVil}.

By using \eqref{eq:quo} and that $I_0(h)=-A(h)$ we get that
\[
\frac{T(h)}{A(h)}=\frac1{(6h-1)h}\left(5h-1-\frac76\frac{I_1(h)}{I_0(h)}\right)=\frac1{(6h-1)h}\left(5h-1-\frac76\frac{v(h)}{w(h)}\right)
\]
would be a rational function. This is a contradiction unless $v=w=0$
because when $h\uparrow1/6$ the left-hand side  of the above
equality goes to infinity as $k\ln(1-6h)$  and the right-hand only
can go to infinity with speed $c(1-6h)^{-m}$ for some $0<m\in N$ and
$c\in\R.$

(ii) From Proposition \ref{pr:AT} we know that all the Abelian
integrals $I_j$ can be expressed in terms of polynomials of $h$,
$T(h)$ and $A(h).$ Similarly in Proposition \ref{pr:pf} we get a
similar property but changing $A(h)$ and $T(h)$ by $I_0(h)=-A(h)$
and $I_1(h).$ Moreover $A'(h)=T(h)$ and $I_1(h)$ can be obtained
from $I_0(h)$ and $I_0'(h)$, see equation \eqref{eq:I1}. In any
case, given the Taylor series at $h=0$ of any of the following
functions
\[
T(h),\, A(h) \quad \mbox{or}\quad I_0(h)
\]
the Taylor series at $h=0$ of all the other function $I_j(h),$ $0\le
j$ can be easily obtained by using the results of these propositions
and, equivalently the Taylor series of all the $F_j(h),$ also at
$h=0.$   For convenience we introduce the function $G_j(h)=
F_j(h)/(\pi h)$  for all $j\ge0$ and the the expression of
\eqref{eq:Fs} for $h\in(0,1/6)$ can be written as
\begin{equation*}
\frac{  I(h)}{\pi h}=\sum_{j=0}^{N(n)} d_j G_j(h),
\end{equation*}
for arbitrary real $d_j.$ For instance, from the expression of
$T(h)$ given in \eqref{eq:th} given in Subsection \ref{ss:apl} or
the one of $I_0(h)$ we obtain that
\begin{align*}
G_{0}(h) =&  -2 - \dfrac{5}{6}  h - \dfrac{385}{216}  h^2 -
\dfrac{85085}{15552}  h^3  - \dfrac{7436429}{373248}
 h^4 - \dfrac{1078282205}{13436928}  h^5  +O(h^6),\\
G_1(h) =&  h + \dfrac{35}{18} h^2 + \dfrac{5005}{864} h^3 +
\dfrac{323323}{15552}h^4 + \dfrac{185910725}{2239488}h^5 +
 \dfrac{4775249765}{13436928}h^6  +O(h^7),\\
G_2(h) =&  \dfrac{5}{3} h^2 + \dfrac{385}{72}h^3+ \dfrac{17017}{864}h^4 + \dfrac{7436429}{93312}h^5+  \dfrac{770201575}{2239488}h^6 +O(h^7),\\
G_3(h) =& -h ^2 - \dfrac{35}{8} h^3- \dfrac{5005}{288} h^4 -
\dfrac{2263261}{31104} h^5-  \dfrac{26558675}{82944}
h^6+O(h^7),\\
G_4(h) =&   -\dfrac{5}{4} h^3- \dfrac{77}{8} h^4-\dfrac{85085}{1728}
h^5 - \dfrac{7436429}{31104} h^6- \dfrac{770201575}{663552}
h^7+O(h^8),
    \end{align*}
and similarly we have obtained all $G_j(h)$ for $0\le j\le N(50)=
33$ until order $50.$ Of course, we do not explicite them. To prove
that the function $(G_j(h))_{j=0}^{N(50)}$ are an ECT system in a
neighborhood $(0,h_1)$ of $h=0, $ from Lemma
\ref{lemasistemaTintervaloaberto} it suffices to prove that the
following Wronskians
$$ W_k(h)= W_k(G_0,\ldots,G_k)(h) =  \left|\begin{array}{ccc}
    G_0(h) & \cdots & G_k(h) \\
    G_0'(h) &  \cdots & G_k'(h) \\
    \vdots &  \ddots & \vdots \\
    G_0^{(k)}(h)  & \cdots & G_k^{(k)}(h)
    \end{array} \right|,$$
do not vanish at $h=0,$ for $k=0,1,\dots,N(50).$

After some tedious calculations we get that $W_k(0)\ne0$ for all
these values of $k.$ For instance, $W_0(0)=-2,$ $W_1(0)=-2,$
$W_2(0)=-20/3,$ $W_3(0)=140/3,$ $W_4(0)=-12320/3,$
\[
W_5(0)= - \dfrac{11211200}{9},\,\, W_6(0)= -
\dfrac{83859776000}{9},\,\,W_7(0)= \dfrac{2899871054080000}{9},
\]
and so on. The result on the $G_j's$ implies the one for the $F_j's$
and, as a consequence, the desired result for the $I_j, j\not\equiv
2 \pmod 3.$

 (iii)   We will apply Theorem \ref{criteriointabeliana} when
    $A(x)=x^2/2+x^3/3$ to the Abelian integrals $I_0(h)$ and $I_1(h).$ By this
    result it suffices  to prove  that $(\ell_0,\ell_1)$ is an $ECT$-system in $\left(0,\frac{1}{2}\right)$,
     where
    \begin{equation*}
    \ell_0(x)= \dfrac{1}{x(1+ x)} - \dfrac{1}{\sigma(x)(1+ \sigma(x))}
    \quad \mbox{and}\quad \ell_1(x)=  \dfrac{1}{1+ x} - \dfrac{1}{1+ \sigma(x)}.
    \end{equation*}
    Derivating implicitly $S(x,\sigma(x))=0,$ where $S$ is given in
    \eqref{eq:W}, we get that
    \[
    \sigma'(x)=-\frac{4x+2z+3}{2x+4z+3},
    \]
where $z=\sigma(x).$ Moreover their  Wronskians are
    \begin{equation*}
    \begin{aligned}
    & W_0(\ell_0)(x) = \ell_0(x) = \dfrac{ (1+x+z)(z -x)}{x(1+x)z(1+z)}, \\
    & W_1(\ell_0, \ell_1)(x) = \dfrac{(z-x)^3
    (4x^2+6xz+4z^2+7x+7z+3)}{x^2(1+x)^2z^2(1+z)^2(2x+4z+3)}.
    \end{aligned}
    \end{equation*}
We will prove that all factors do not vanish when $x\in(0,1/2).$ For
the first function this holds trivially for all factors but one,
$1+x+z=1+x+\sigma(x)$ which is precisely the one that we have
studied in detail in Section \ref{ss:param} with our approach using
rational parameterizations, see equation \eqref{eq:exx}.

Let us study the remaining factors $R_2(x,z)=4x^2+6xz+4z^2+7x+7z+3$
and $R_1(x,y)=2x+4z+3.$ For them it suffices to use the resultants
approach, also explained in Section \ref{ss:param}. It holds that
\[
\operatorname{Res}_z\big(R_2(x,z),S(x,z)\big)
=\frac49x^4+\frac89x^3-\frac29x^2-\frac23x+\frac12,
\]
and
\[
\operatorname{Res}_z\big(R_1(x,z),S(x,z)\big) =(2x+3)(2x-1),
\]
both not vanishing for $x\in(0,1/2),$ as desired. In fact, for the
first one, Sturm's approach proves that it has no real roots.

(ii) We want to use the same approach that in item (i) to prove that
the functions $I_0,I_1$ and $I_3$ form a Chebyshev system, but we
must reorder them because, otherwise our approach fails. We will
prove that $(I_3(h),I_1(h),I_0(h))$ form a Chebyshev system on
$(0,1/6),$ or equivalently that the functions
$(hI_3(h),hI_1(h),hI_0(h))$ form an ECT. By item (ii) of Lemma
\ref{lemaexpoente} we have that
\begin{align*}
hI_0(h)=& h \int_{\gamma_h} y\,dx= \int_{\gamma_h}f_0(x)
y^3\,dx,\quad\mbox{with} &&f_0(x)=
\frac{11x^2+22x+12}{18(1+x)^2},\\
hI_1(h)=& h \int_{\gamma_h} x y\,dx= \int_{\gamma_h}f_1(x)
y^3\,dx,\quad\mbox{with} &&f_1(x)=
\frac{x(13x^2+27x+15)}{18(1+x)^2},\\
hI_3(h)=& h \int_{\gamma_h} x^3 y\,dx=
\int_{\gamma_h}f_3(x)y^3\,dx,\quad\mbox{with} &&f_3(x)=
\frac{x^3(17x^2+37x+21)}{18(1+x)^2}.\\
\end{align*}
To apply Theorem \ref{criteriointabeliana}, first  we consider the
functions
\[
\ell_i(x)=\frac{f_i(x)}{A'(x)}-\frac{f_i(\sigma(x))}{A'(\sigma(x))},
\quad i=0,1,3.
\]
and compute the following Wronskians, where $z=\sigma(x),$
\begin{align*}
& W_0(\ell_3)(x) = \ell_3(x) = \dfrac{(x-z)R_6(x,z)}{18(1+x)^3(1+z)^3}, \\
    & W_1(\ell_3, \ell_1)(x) = \dfrac{(x-z)^3R_7(x,z)}
    {324(2x+4z+3)(1+x)^5(1+z)^5}, \\
      & W_2(\ell_3, \ell_1,\ell_0)(x) = \dfrac{(x-z)^6R_{12}(x,z)}
      {2916(2x+4z+3)^3x^3z^3(1+x)^7(1+z)^7}
\end{align*}
where
\begin{align*}
R_6(x,z)=&17 {x}^{3}{z}^{3}+51 {x}^{3}{z}^{2}+51 {x}^{2}{z}^{3}+51
{x}^{3}z+ 141 {x}^{2}{z}^{2}+51 x{z}^{3}+17 {x}^{3}\\&+128
{x}^{2}z+128 x{z}^{
2}+17 {z}^{3}+37 {x}^{2}+100 xz+37 {z}^{2}+21 x+21 z,\\
\end{align*}
and the polynomials $R_k\in\mathbb{Z}[x,y]$ have  degree $k,$ and we
do not explicit them for the sake of shortness. In the computations
of $W_2$ we have used that
\[
\sigma''(x)=-\frac{12(4x^2+4xz+4z^2+6x+6z+3  )}{(2x+4z+3)^3},
\]
expression obtained once more, derivating implicitly
$S(x,\sigma(x))=0.$ Finally,
\begin{align*}
R_{6}(u(s),\sigma(u(s)))&=R_{6}(u(s),v(s))=-\frac{9s^2(s-4)^2S_8(s)}
{8(s^2-2s+4)^6},\\
R_{7}(u(s),\sigma(u(s)))&=R_{7}(u(s),v(s))=\frac{27(s-4)^2S_{12}(s)}
{8(s^2-2s+4)^6},\\
R_{12}(u(s),\sigma(u(s)))&=R_{12}(u(s),v(s))=\frac{729(s-4)^2S_{22}(s)}
{32(s^2-2s+4)^{12}},
\end{align*}
where
\[
S_8(s)=5 {s}^{8}-55 {s}^{7}+197 {s}^{6}-43 {s} ^{5}-1162
{s}^{4}+1100 {s}^{3}+5240 {s}^{2}-10240 s+5120,
\]
and  $S_k\in\mathbb{Z}[s],$ have degree $k,$ and we do not explicit
$S_{12}$ and $S_{22}.$ By using Sturm's approach we prove that none
of them vanish in $(0,1),$ as desired.
\end{proof}

\begin{proof}[Proof of Corollary \ref{corolcotamaxima01}] From the
results of Subsection \ref{ss:whp} the simple zeroes in $(0,1/6)$ of
\begin{align*}\label{eq:final}
I(h)=\int_{\gamma_h} P(x)\,dy=\iint_{\operatorname{Int}(\gamma_h}
P'(x)\,dxdy=-\int_{\gamma_h} P'(x) y\, dx=\sum_{j=0}^{n} \alpha_j
I_{j}(h)
\end{align*}
give rise to limit cycles of \eqref{sistemaqueaplicamosintabel} for
$\varepsilon$ small enough. By using item (i) of Theorem
\ref{th:main2} the result follows.
\end{proof}

\begin{rem}\label{re:final}
It is known that for many Hamiltonian systems, some $k\times k$
Picard-Fuchs system differential equations is satisfied by several
Abelian integrals, see for instance \cite{NovYak2001} and their
references. When one of these integrals is the area function
$A(h)=-\int_{\gamma_h} y\,dx,$
 from its Taylor's expansion at $h=0$ it is not difficult to get
  the Taylor expansion at $h=0$ of all the other Abelian integrals involved in the system.
  Then the approach used to prove item (ii) of Theorem \ref{th:main2}, that recall reduces to compute
  a Wronskian at $h=0,$ can be applied to
study lower bounds of the number of limit cycles bifurcating from
the periodic orbits of the Hamiltonian. Since $A'(h)=T(h),$ the
Taylor's expansion of $A(h)$ near the center can be obtained either
by using Theorem \ref{th:main1} or by using the linear $k$-th order
differential equation satisfied by $A(h)$ obtained from the
Picard-Fuchs system. Notice that to study system \eqref{eq:final} we
have used both approaches.
\end{rem}

\section*{Acknowlegements}

This work has received funding from the Ministerio de Econom\'{\i}a,
Industria y Competitividad - Agencia Estatal de Investigaci\'{o}n
(MTM2016-77278-P FEDER grant), the Ag\`{e}ncia de Gesti\'{o} d'Ajuts
Universitaris i de Recerca (2017 SGR 1617 grant), CAPES grant
88881.068462/2014-01, CNPq grant 304798/2019-3, and São Paulo Paulo Research Foundation (FAPESP) grants 2019/10269-3, 2018/05098-2 and 2016/00242-2.


\end{document}